\documentclass[11pt]{article}
\usepackage{amsmath}
\usepackage{amssymb}
\usepackage{latexsym}
\usepackage{amsthm}
\usepackage{mathrsfs}
\usepackage{wasysym}
\usepackage{fancyhdr}
\usepackage{amsfonts}
\usepackage{amssymb}%Use Miscell. symbols
\usepackage{epsfig}
\usepackage{caption}

\usepackage{floatrow}

\newtheorem{thm}{Theorem}[section]
\newtheorem{lem}[thm]{Lemma}
\newtheorem{con}[thm]{Conjecture}

\newtheorem{pro}[thm]{Problem}

\begin{document}

\title{Berge-Fulkerson coloring for infinite families of snarks}

\author{Ting Zheng, Rong-Xia Hao\footnote{Corresponding author, rxhao@bjtu.edu.cn}\\
Department of Mathematics, Beijing Jiaotong University,\\
 Beijing, 100044, P.R.China}

\maketitle

\begin{abstract}
It is conjectured by Berge and Fulkerson
that {\em every bridgeless cubic graph
 has six perfect matchings such that each edge is contained in
exactly two of them.} H$\ddot{a}$gglund constructed two graphs Blowup$(K_4, C)$ and Blowup$(Prism, C_4)$.
Based on these two graphs, Chen constructed infinite families of bridgeless cubic graphs $M_{0,1,2, \ldots,k-2, k-1}$ which is obtained from cyclically 4-edge-connected and having a Fulkerson-cover cubic graphs $G_0,G_1,\ldots, G_{k-1}$ by recursive process. If each $G_i$ for $1\leq i\leq k-1$ is a cyclically 4-edge-connected snarks with excessive index at least 5, Chen proved that these infinite families are snarks. He obtained that each graph in $M_{0,1,2,3}$ has a Fulkerson-cover and gave the open problem that whether every graph in $M_{0,1,2, \ldots,k-2, k-1}$ has a Fulkerson-cover. In this paper, we solve this problem and prove that every graph in $M_{0,1,2, \ldots,k-2, k-1}$ has a Fulkerson-cover.

{\bf Keyword}: Berge-Fulkerson conjecture, Fulkerson-cover, perfect matching.

{\bf MSC(2010)}: Primary: 05C70; Secondary: 05C75, 05C40, 05C15.
\end{abstract}

\section{Introduction}

Let $G$ be a simple graph with vertex set $V(G)$ and edge set $E(G)$.
A {\em circuit} of $G$ is a $2$-regular connected subgraph. An {\em even graph} is a graph with even degree at every vertex.
 A perfect matching of $G$ is a 1-regular spanning subgraph of $G$. The excessive index of $G$, denoted by $\chi_e'(G)$, is the least integer $k$, such that $G$ can be covered by $k$ perfect matchings. A cubic graph is a {\em snark} if it is bridgeless and not $3$-edge-colorable. A cubic graph $G$ is {\em Berge-Fulkerson colorable} if $2G$ is $6$-edge-colorable.
It is an equivalent description of the Berge-Fulkerson conjecture.

  The following is a famous open problem called Berge-Fulkerson conjecture.

\begin{con}
{\rm (Berge-Fulkerson Conjecture \cite{D}, or see \cite{S})}
Every bridgeless cubic graph has six perfect matchings such that each edge belongs to exactly two of them.
\end{con}

%We call such 6 perfect in the conjecture matchings as the Fulkerson-cover.

Although there are some results related with this conjecture, as examples, see~\cite{Chen2015},\cite{Fan1994}, \cite{J},\cite{Macajova2009},\cite{Mazzuoccolo2011}, Berge-Fulkerson conjecture is still open for many bridgeless cubic graphs even for some simple snarks.

H$\ddot{a}$gglund~\cite{J} constructed two graphs Blowup$(K_4, C)$ and Blowup$(Prism, C_4)$. Based on Blowup$(K_4, C)$ , Esperet et al.~\cite{L} constructed infinite families of cyclically 4-edge-connected snarks with excessive index at least five. Based on these two graphs, Chen~\cite{F} constructed infinite families of cyclically 4-edge-connected snarks $E_{0,1,2,\ldots,(k-1)}$ obtained from
 cyclically 4-edge-connected snarks $G_0,G_1,\ldots, G_{k-1}$, in which $E_{0,1,2}$ is Esperet et al.' construction. If only assume that each graph in  $\{G_0,G_1,\ldots, G_{k-1}\}$ has a Fulkerson-cover, then these infinite families of bridgeless cubic graphs are denoted by $M_{0,1,2, \ldots,k-2, k-1}$. Chen~\cite{F} obtained that every graph in $M_{0,1}$ has a Fulkerson-cover and each graph in $M_{0,1,2,3}$ has a Fulkerson-cover and gave the following problem.

\begin{pro}\label{prolem-1}\cite{F}
If $H=\{G; G_0, G_1, \ldots,G_{k-2}, G_{k-1}\} \in M_{0,1,2, \ldots,k-2, k-1}$, Does $H$ have a Fulkerson-cover?
\end{pro}

In this paper, we solve Problem~\ref{prolem-1}. The main result is Theorem~\ref{Main thm}.

 \begin{thm}\label{Main thm}
Each graph in $M_{0, 1, 2, \ldots, k-2, k-1}$ for $k>2$ has a Fulkerson-cover.
\end{thm}

\section{Preliminaries}

In this section, some necessary definitions, constructions and Lemma are given.

Let $X\subseteq V(G)$ and $Y\subseteq E(G)$. We use $G\setminus X$ to denote the subgraph of $G$ obtained from $G$ by deleting all the vertices of $X$ and all the edges incident with $X$. While $G\setminus Y$ to denote the subgraph of $G$ obtained from $G$ by deleting all the edges of $Y$. The edge-cut of $G$ associated with $X$, denoted by $\partial_{G}(X)$, is the set of edges of $G$ with exactly one end in $X$. The edge set $C=\partial_{G}(X)$ is called a $k$-edge-cut if $|\partial_{G}(X)|=k$.  A cycle of $G$ is a subgraph of $G$ with each vertex of even degree. A circuit of $G$ is a minimal 2-regular cycle of $G$. A graph $G$ is called cyclically $k$-edge-connected if at least k edges must be removed to disconnect it into two components, each of which contains a circuit.

Let $G_i$ be a cyclically 4-edge-connected snark with excessive index at least 5, for $i=0,1$. Let $x_iy_i$ be an edge of $G_i$ and $x_i^0$, $x_i^1$ ($y_i^0$, $y_i^1$) be the neighbours of $x_i$ ($y_i$). Let $H_i$ be the graph obtained from $G_i$ by deleting the vertices $x_i$ and $y_i$. Let $\{G; G_0, G_1\}$ be the graph obtained from the disjoint union of $H_0$, $H_1$ by adding six vertices $a_0$, $b_0$, $c_0$, $a_1$, $b_1$, $c_1$ and 13 edges $a_0y_0^0$, $a_0x_1^0$, $a_0c_0$, $c_0b_0$, $b_0y_0^1$, $b_0x_1^1$, $b_1x_0^1$, $b_1y_1^1$, $b_1c_1$, $c_1a_1$, $a_1x_0^0$, $a_1y_1^0$, $c_0c_1$. The graphs of this type are denoted as $E_{0,1}$ (see Figure~\ref{F1}).

\begin{figure}[ht]
\hskip5cm
\begin{center}
\vskip-1.5cm
\includegraphics[scale=0.5]{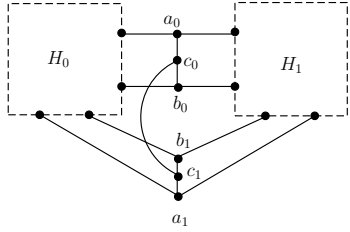}
\end{center}
\vskip-10cm
\caption{$\{G;G_0, G_1\}$}\label{F1}
\end{figure}

 The family of graphs $E_{0,1, \ldots,(k-1)}$ ($k\geq 2$) and $M_{0,1, \ldots,(k-1)}$ ($k\geq 2$) are constructed by Chen as follows:

(1) $\{G; G_0, G_1\}\in E_{0,1}$ with $A_j=\{a_j, b_j, c_j\}$ for $j=0, 1$.

(2) For $3\leq i\leq k$, $\{G; G_0, G_1, \ldots,G_{i-1}\}$ is obtained from $\{G; G_0, G_1, \ldots,$ $G_{i-2}\}\in E_{0,1, \ldots,(i-2)}$ by adding $H_{i-1}$ and $A_{i-1}=\{a_{i-1}, b_{i-1}, c_{i-1}\}$ and by inserting a vertex $v_{i-3}$ into $e_0$, such that

(i) $G_{i-1}$ is a cyclically 4-edge-connected snark with excessive index at least 5 ($x_{i-1}y_{i-1}$ is an edge of $G_{i-1}$ and $x_{i-1}^0$, $x_{i-1}^1$ ($y_{i-1}^0$, $y_{i-1}^1$) are the neighbours of $x_{i-1}$ ($y_{i-1}$));

(ii) $H_{i-1}=G_{i-1}\backslash \{x_{i-1}, y_{i-1}\}$;

(iii) $e_0\in E(\{G; G_0, G_1, \ldots, G_{i-2}\})-\bigcup\limits_{j=0}^{i-2} E(H_j)-\bigcup\limits_{j=0}^{i-2}\{a_jc_j, c_jb_j\}$ and $e_0$ is incident with $c_0$;

(iv) $a_{i-1}$ is adjacent to $x_0^0$ and $y_{i-1}^0$, $b_{i-1}$ is adjacent to $x_0^1$ and $y_{i-1}^1$, $a_{i-2}$ is adjacent to $x_{i-1}^0$ and $y_{i-2}^0$, $b_{i-2}$ is adjacent to $x_{i-1}^1$ and $y_{i-2}^1$, $c_{i-1}$ is adjacent to $a_{i-1}$, $b_{i-1}$ and $v_{i-3}$, the other edges of $\{G; G_0, G_1, \ldots ,G_{i-2}\}$ remain the same;

(v) $\{G; G_0, G_1, \ldots,G_{i-1}\}\in E_{0,1, \ldots,(i-1)}$.

The class of graphs constructed by Esperet et al. is a special case for $k=3$ of $E_{0,1, \ldots,(k-1)}$. If the excessive index and non 3-edge-colorability of $G_i (i=0, 1, 2,\ldots , (k-1))$ is ignored and only assume that $G_i$ has a Fulkerson-cover, then we obtain infinite families of bridgeless cubic graphs. We denote graphs of this type as $M_{0,1,2, \ldots,(k-1)} (k\geq 2)$.

The following Lemma is very import in our main proofs.

\begin{lem}\label{lemma-1}
{\rm (Hao, Niu, Wang, Zhang and Zhang \cite{H})}
A bridgeless cubic graph $G$ has a Fulkerson-cover if and only if there are two disjoint matchings $E_0$ and $E_2$, such that $E_0\cup E_2$ is a cycle and $\overline{G\setminus
 E_i}$ is 3-edge colorable, for each $i=0, 2$.
\end{lem}

%We prove that every graph in $M_{0,1}$ has a Fulkerson-cover with more detailed proofing and solve the Problem~\ref{prolem-1} in this article.

\section{Each graph in $M_{0, 1, 2, \ldots, k-2, k-1}$ has a Fulkerson cover}
We give the results according to the parity of $k$.

\begin{thm}\label{thm-2} Let $k$ be an even integer and $k\geq 4$.
If $\Gamma\in M_{0,1,2, \ldots,k-2, k-1}$, then $\Gamma$ has a Fulkerson-cover.
\end{thm}

\begin{proof} Since $\Gamma\in M_{0,1,2, \ldots,k-2, k-1}$, assume $\Gamma=\{G; G_0, G_1, \ldots,G_{k-2}, G_{k-1}\}$.

Since $G_{i}$ has a Fulkerson-cover, for each $i=0,1,\ldots,k-1$, suppose that $\{M_i^1, M_i^2, M_i^3$ $, M_i^4, M_i^5, M_i^6\}$ is the Fulkerson-cover of $G_{i}$. Let $E_i^{2}$ be the set of edges in $G_i$ covered twice by $\{M_i^1, M_i^2, M_i^3 \}$ and $E_i^{0}$ be the set of edges in $G_i$ which are not covered by $\{M_i^1, M_i^2, M_i^3 \}$. Note that $E_i^{2} \cup E_i^{0}$ is an even cycle, and $\overline{G_i\setminus E_i^{2}}$ and $\overline{G_i\setminus E_i^{0}}$ can be colored by three colors. Then $E_i^{2}$ and $E_i^{0}$ are the desired disjoint matchings of $G_i$ as in Lemma ~\ref{lemma-1}. By choosing three perfect matchings of $G_{i}$, for each $i=0, 1,\ldots, k-1$, we can obtain two desired disjoint matchings $E_i^{2}$ and $E_i^{0}$ such that $x_iy_i\in E_i^{2}\cup E_i^{0}$ or $x_{i}, y_i\not\in V(E_i^{2}\cup E_i^{0})$.

Three perfect matchings $\{M_i^1, M_i^2, M_i^3 \}$ of $G_i$ are chosen such that $x_{i}, y_i\not\in V(E_i^{2} \cup E_i^{0})$ if $i$ is even; And three perfect matchings $\{M_i^1, M_i^2, M_i^3 \}$ of $G_i$ are chosen such that $x_iy_i\in E_i^{2}\cup E_i^{0}$ if $i$ is odd. Without loss of generality, assume that $x_iy_i\in E_i^{2}$ and $x_i^0x_i, y_i^0y_i\in E_i^{0}$ for odd $i$.

\vskip0.5cm

\begin{figure}[ht]
\begin{center}
\vskip-1cm
\includegraphics[scale=0.45]{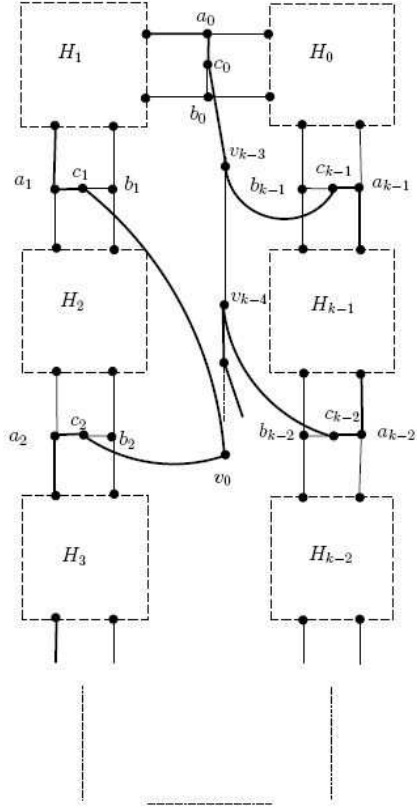}
\vskip-20cm
\end{center}
\caption{$\{G; G_0, G_1, \ldots,G_{k-1}\}$}\label{F6}
\end{figure}

\begin{figure}[ht]
\begin{center}
\vskip-3cm
\hskip-3cm
\includegraphics[scale=0.7]{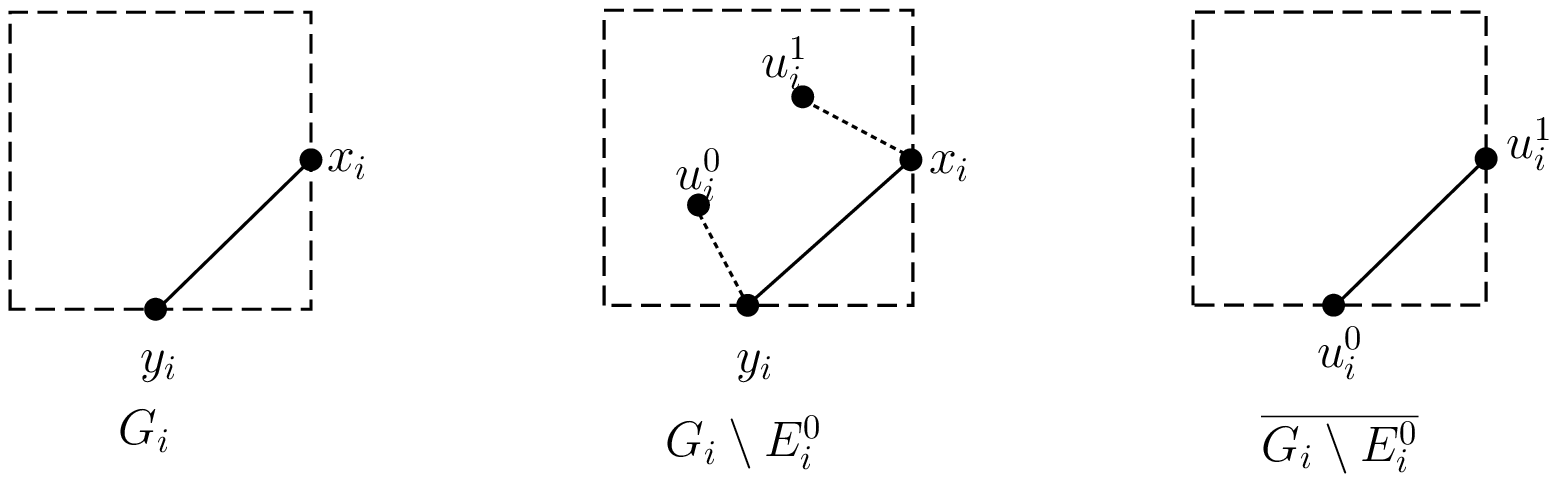}
\end{center}
\vskip-14cm
\caption{}\label{F2}
\end{figure}

\begin{figure}[ht]
\begin{center}
\vskip-3cm
\hskip-3cm
\includegraphics[scale=0.7]{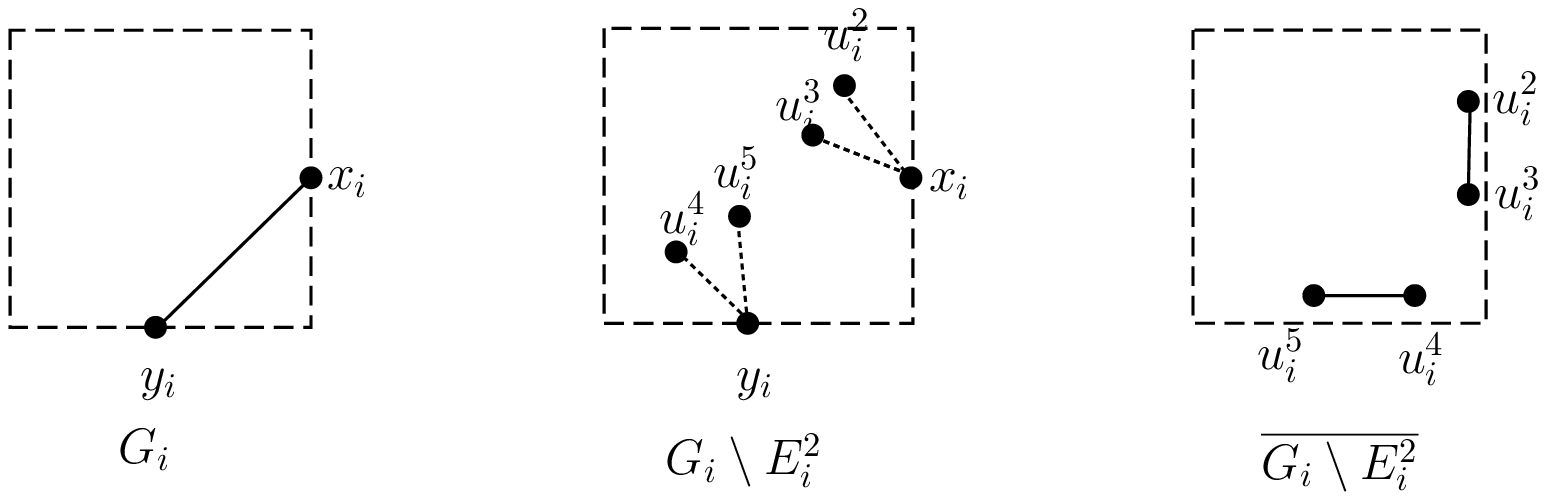}
\end{center}
\vskip-14cm
\caption{}\label{F3}
\end{figure}

Let $$E_0=(E_0^2-x_0y_0)\cup (E_1^0-\{x_1^0x_1, y_1^0y_1\})\cup \bigcup\limits_{i=2}^{k-1}(E_i^2-{x_iy_i})\cup \bigcup\limits_{i=2}^{k-1}a_ic_i$$ $$\cup \bigcup\limits_{j=1}^{\frac{k-4}{2}} v_{2j-1}v_{2j}\cup \{c_0v_{k-3}, c_1v_0, y_1^0a_1, x_1^0a_0\}$$ and $$E_2=(E_0^0-\{x_0x_0^0,y_0y_0^0\})\cup (E_1^2-{x_1y_1})\cup \bigcup\limits_{i=2}^{k-1}(E_i^0-\{x_ix_i^0, y_iy_i^0\})$$ $$\cup \bigcup\limits_{j=1}^{\frac{k-2}{2}}\{y_{2j+1}^0a_{2j+1}, x_{2j+1}^0a_{2j}\}\cup \bigcup\limits_{i=2}^{k-1} v_{i-2}c_i\cup \{a_0c_0, a_1c_1\}.$$ Clearly, $E_0\cup E_2$ is an even cycle $C$. (See Figure~\ref{F6}.)

If $i$ is odd, by $x_iy_i\in E_i^{2}$, there exists a maximal path containing only 2-degree vertices as inter vertices in the graph $G_i\setminus E_i^0$, say $u_i^0\ldots y_ix_i\ldots u_i^1$, which corresponds to an edge $u_i^0u_i^1$ in the graph $\overline{G_i\setminus E_i^0}$ (see Figure~\ref{F2}). From $x_i^0x_i, y_i^0y_i\in E_i^{0}$, there exists two maximal paths containing only 2-degree vertices as inter vertices in the graph $G_i\setminus E_i^2$, say $u_i^2\ldots x_i^0x_ix_i^1\ldots u_i^3$ and $u_i^5\ldots y_i^0y_iy_i^1\ldots u_i^4$, which corresponds to an edge $u_i^2u_i^3$ and $u_i^4u_i^5$, respectively, in the graph $\overline{G_i\setminus E_i^2}$ (see Figure~\ref{F3}).

If $i$ is even, by $x_{i}, y_i\not\in V(E_i^{2})$, there exist four maximal paths containing only 2-degree vertices as inter vertices in the graph $G_i\setminus E_i^2$, say $u_i^0\ldots x_i^1x_i$ (maybe $u_i^0=x_i^1$), $u_i^1\ldots x_i^0x_i$ (maybe $u_i^1=x_i^0$), $u_i^2\ldots y_i^1y_i$ (maybe $u_i^2=y_i^1$) and $u_i^3\ldots y_i^0y_i$, which correspond to four edges $u_i^0x_i$, $u_i^1x_i$, $u_i^2y_i$ and $u_i^3y_i$, respectively, in the graph $\overline{G_i\setminus E_i^2}$ (see Figure~\ref{F4}).

Similarly, by $x_{i}, y_i\not\in V(E_i^{0})$, there exist four maximal paths containing only 2-degree vertices as inter vertices in the graph $G_i\setminus E_i^0$, say $u_i^4\ldots x_i^1x_i$ (maybe $u_i^4=x_i^1$), $u_i^5\ldots x_i^0x_i$ (maybe $u_i^5=x_i^0$), $u_i^6\ldots y_i^1y_i$ (maybe $u_i^6=y_i^1$) and $u_i^7\ldots y_i^0y_i$ (maybe $u_i^7=y_i^0$), which correspond to four edges $u_i^4x_i$, $u_i^5x_i$, $u_i^6y_i$ and $u_i^7y_i$, respectively, in $\overline{G_i\setminus E_i^0}$ (see Figure~\ref{F5}).

\begin{figure}[ht]
\begin{center}
\vskip-3cm
\hskip-3cm
\includegraphics[scale=0.7]{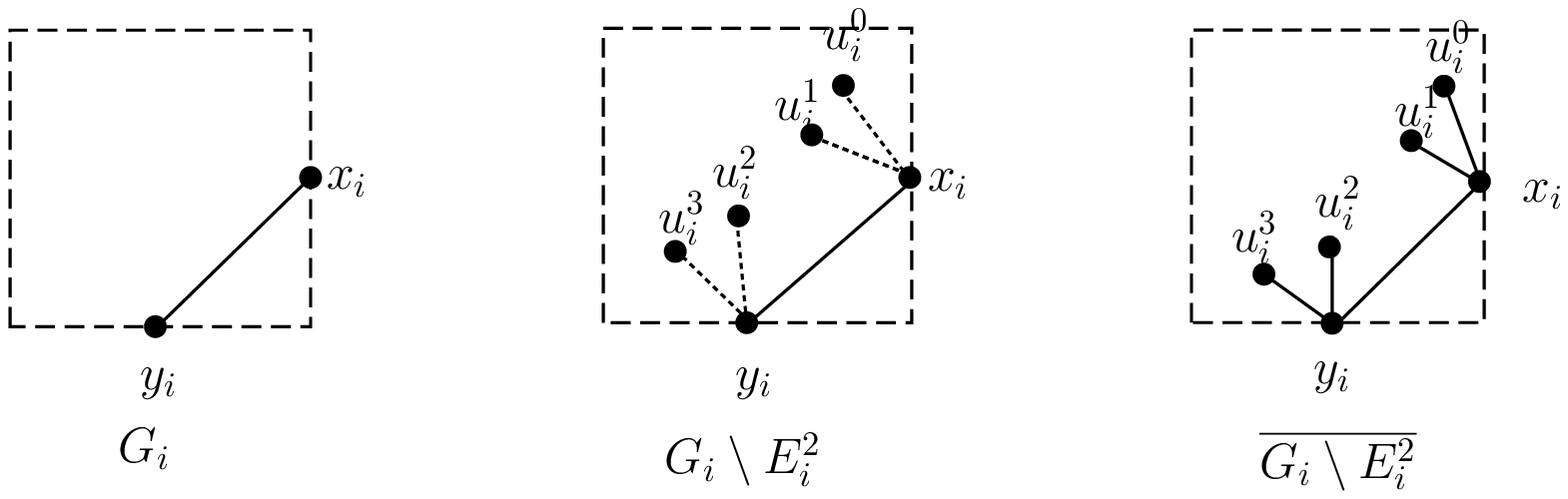}
\end{center}
\vskip-14cm
\caption{}\label{F4}
\end{figure}

\begin{figure}[ht]
\begin{center}
\vskip-3cm
\hskip-3cm
\includegraphics[scale=0.7]{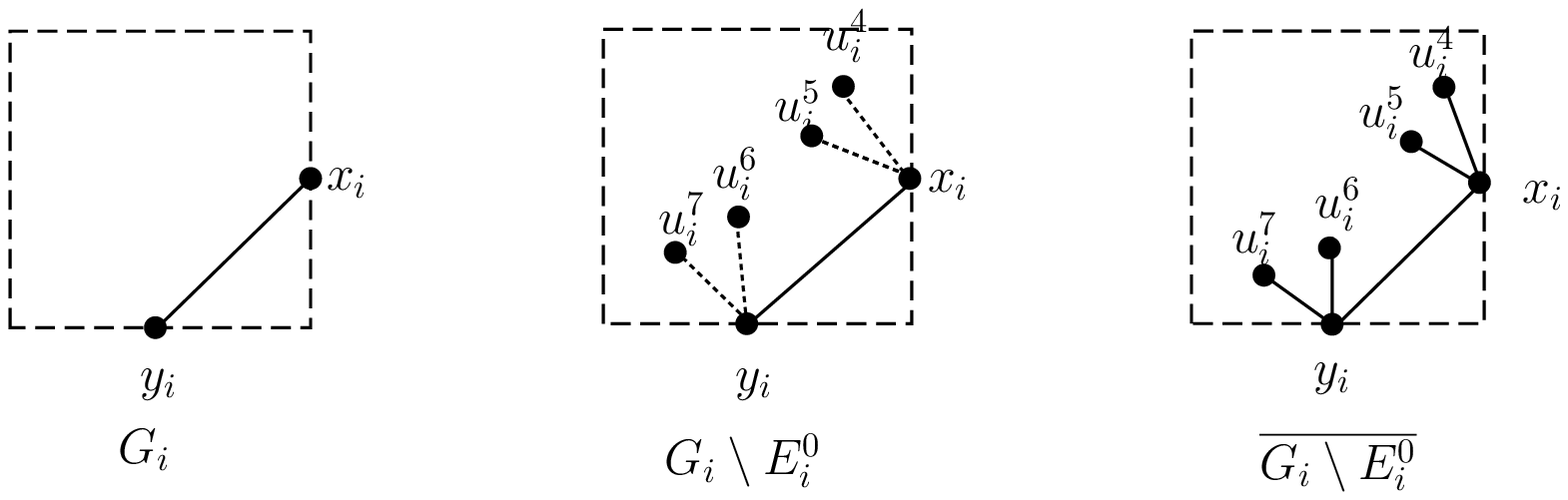}
\end{center}
\vskip-14cm
\caption{}\label{F5}
\end{figure}

\begin{figure}[ht]
\begin{center}
\vskip-0.8cm
\includegraphics[scale=0.4]{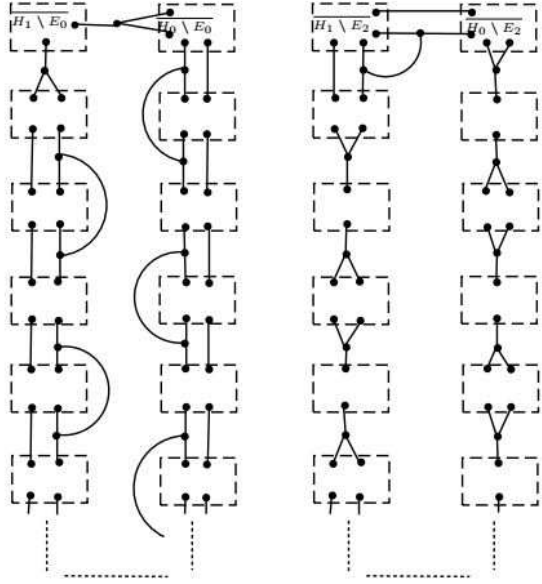}
\vskip-20cm
\end{center}
\caption{$\overline{\{G;G_0,G_1,\ldots,G_{k-1}\}\setminus E_0}$ and $\overline{\{G;G_0,G_1,\ldots,G_{k-1}\}\setminus E_2}$}\label{F7}
\end{figure}

 From the construction of $\Gamma$, we know that  $\overline{\Gamma\setminus E_0}$ (see Figure~\ref{F7}) is

 $$\begin{array}{l}
(\overline{G_1\setminus E_1^0}-u_1^0u_1^1)\cup \bigcup\limits_{j=0}^{\frac{k-2}{2}}(\overline{G_{2j}\setminus E_{2j}^2}-\{x_{2j}, y_{2j}\})\\
 \cup \bigcup\limits_{j=1}^{\frac{k-2}{2}}(\overline{G_{2j+1}\setminus E_{2j+1}^2}-\{u_{2j+1}^2u_{2j+1}^3, u_{2j+1}^4u_{2j+1}^5\})\cup \{u_1^0b_1, u_1^1b_0, b_1u_2^0, b_1u_2^1, u_0^2b_0, u_0^3b_0\}\\
 \cup\bigcup\limits_{j=1}^{\frac{k-2}{2}} \{u_{2j}^3u_{2j+1}^2, u_{2j}^2b_{2j}, u_{2j+1}^3b_{2j}, u_{2j+1}^5u_{2j+2}^1, u_{2j+1}^4b_{2j+1}, u_{2j+2}^0b_{2j+1}, b_{2j}b_{2j+1}\}\end{array}$$

  and $\overline{\Gamma\setminus E_2}$ (see Figure~\ref{F7}) is
  $$\begin{array}{l}(\overline{G_1\setminus E_1^2}-\{u_1^2u_1^3, u_1^4u_1^5\})\cup \bigcup\limits_{j=0}^{\frac{k-2}{2}}(\overline{G_{2j}\setminus E_{2j}^0}-\{x_{2j}, y_{2j}\})\\
  \cup \bigcup\limits_{j=1}^{\frac{k-2}{2}} (\overline{G_{2j+1}\setminus E_{2j+1}^0}-u_{2j+1}^0u_{2j+1}^1)\\
  \cup \{u_1^4b_1, u_1^5u_2^5, b_1u_2^4, u_1^3b_0, u_0^7u_1^2, u_0^6b_0, b_0b_1\}\\
  \cup\bigcup\limits_{j=1}^{\frac{k-2}{2}}\{u_{2j}^6b_{2j}, u_{2j}^7b_{2j}, b_{2j}u_{2j+1}^1, u_{2j+1}^0b_{2j+1}, b_{2j+1}u_{2j+2}^4, b_{2j+1}u_{2j+2}^5\}.\end{array}$$

 If $i$ is odd, because $E_i^{2}$ and $E_i^{0}$ are the desired disjoint matchings of $G_i$ as in Lemma ~\ref{lemma-1}, $\overline{G_i\setminus E_i^0}$ is 3-edge colorable. Thus there exists a 2-factor, say $C_i^0$, such that each component is an even circuit and $u_i^0u_i^1$ is not in the 2-factor $C_i^0$. Similarly, because $\overline{G_i\setminus E_i^2}$ is 3-edge colorable, there exists a 2-factor $C_i^2$ such that each component is an even circuit and $\{u_i^2u_i^3, u_i^4u_i^5\}$ is in the 2-factor $C_i^2$.

 If $i$ is even, because $\overline{G_i\setminus E_i^2}$ is 3-edge colorable, there exists a 2-factor $C_i^2$ such that each component is an even circuit and $u_i^0x_iu_i^1$ and $u_i^2y_iu_i^3$ are in the 2-factor $C_i^2$. Because $\overline{G_i\setminus E_i^0}$ is 3-edge colorable, there exists a 2-factor $C_i^0$ such that each component is an even circuit and $\{u_i^4x_iu_i^5, u_i^6y_iu_i^7\}$ is in the 2-factor $C_i^0$.

 Then $\overline{\Gamma\setminus E_0}$ has a 2-factor:
 $$\begin{array}{l} C_1^0\cup \bigcup\limits_{j=1}^{\frac{k-2}{2}}(C_{2j+1}^2-\{u_{2j+1}^2u_{2j+1}^3, u_{2j+1}^4u_{2j+1}^5\})\cup \bigcup\limits_{j=0}^{\frac{k-2}{2}}(C_{2j}^2-\{x_{2j}, y_{2j}\})\cup \{b_0u_0^3, b_0u_0^2,\\
  b_1u_2^0, b_1u_2^1\}\cup \bigcup\limits_{j=1}^{\frac{k-2}{2}}\{b_{2j}u_{2j+1}^3, b_{2j}u_{2j}^2, u_{2j}^3u_{2j+1}^2, b_{2j+1}u_{2j+1}^4, b_{2j+1}u_{2j+2}^0, u_{2j+1}^5u_{2j+2}^1\}.\end{array}$$ And each component is an even circuit.

 And $\overline{\Gamma\setminus E_2}$ has a 2-factor:
 $$\begin{array}{l}\bigcup\limits_{j=1}^{\frac{k-2}{2}}C_{2j+1}^0\cup (C_1^2-\{u_1^2u_1^3, u_1^4u_1^5\})\cup \bigcup\limits_{j=0}^{\frac{k-2}{2}}(C_{2j}^0-\{x_{2j}, y_{2j}\})\cup \{b_0u_0^6, b_0u_1^3, u_0^7u_1^2,\\ b_1u_1^4,
 u_1^5u_2^5, b_1u_2^4\}\cup \bigcup\limits_{j=1}^{\frac{k-2}{2}}\{u_{2j}^6b_{2j}, u_{2j}^7b_{2j}, u_{2j+2}^4b_{2j+1}, u_{2j+2}^5b_{2j+1}\}.\end{array}$$ And each component is an even circuit.

 So $\overline{\Gamma\setminus E_0}$ and $\overline{\Gamma\setminus E_2}$ are 3-edge colorable. Therefore $E_0$ and $E_2$ are the desired matchings in $\Gamma$ of Lemma~\ref{lemma-1}. So $\Gamma=\{G; G_0, G_1,$ $\ldots,G_{k-2},$ $G_{k-1}\}$ has a Fulkerson-cover.
\end{proof}

\begin{thm}\label{thm-3} Let $k$ be an odd integer. If $H\in M_{0, 1, 2, \ldots, k-2, k-1}$, then $H$ has a Fulkerson-cover.
\end{thm}

\begin{proof} By $H\in M_{0, 1, 2, \ldots, k-2, k-1}$, assume $H=\{G; G_0, G_1, \ldots,G_{k-2}, G_{k-1}\}$.
Since $G_{i}$ has a Fulkerson-cover, for each $i=0,1,\ldots,k-1$, suppose that $\{M_i^1, M_i^2, M_i^3 $ $, M_i^4, M_i^5, M_i^6\}$ is the Fulkerson-cover of $G_{i}$. Let $E_i^{2}$ be the set of edges covered twice by $\{M_i^1, M_i^2, M_i^3 \}$ and $E_i^{0}$ be the set of edges not covered by $\{M_i^1, M_i^2, M_i^3 \}$. Now $E_i^{2} \cup E_i^{0}$ is an even cycle, and $\overline{G_i\setminus E_i^{2}}$ and $\overline{G_i\setminus E_i^{0}}$ can be colored by three colors. Then $E_i^{2}$ and $E_i^{0}$ are the desired disjoint matchings of $G_i$ as in Lemma~\ref{lemma-1}. By choosing three perfect matchings of $G_{i}$, for each $i=0,1,\ldots,k-1$, we can obtain two desired disjoint matchings $E_i^{2}$ and $E_i^{0}$ such that $x_{i}y_i\in E_i^{2}\cup E_i^{0}$ or $x_{i}, y_i\not\in V(E_i^{2} \cup E_i^{0})$. If $i$ is even and $i\neq 0$, three perfect matchings of $G_i$ are chosen such that $x_{i}, y_i\not\in V(E_i^{2} \cup E_i^{0})$. If $i$ is odd or $i=0$, three perfect matchings of $G_i$ are chosen such that $x_{i}y_i\in E_i^{2}\cup E_i^{0}$. Without loss of generality, assume that $x_iy_i\in E_i^{2}$ and $x_i^0x_i, y_i^0y_i\in E_i^{0}$ for $i=0$ or $i$ is odd.

Let $$E_0=(E_0^2-x_0y_0)\cup (E_1^0-\{x_1^0x_1, y_1^0y_1\})\cup \bigcup\limits_{i=2}^{k-1}(E_i^2-{x_iy_i})\cup \{y_1^0a_1, x_1^0a_0, c_1v_0\}$$ $$\cup \bigcup\limits_{i=2}^{k-1}a_ic_i\cup \bigcup\limits_{j=0}^{\frac{k-5}{2}}v_{2j+1}v_{2j+2}$$
and $$E_2=(E_1^2-{x_1y_1})\cup \bigcup\limits_{i=2}^{k-1} (E_i^0-\{x_ix_i^0, y_iy_i^0\})\cup(E_0^0-\{x_0x_0^0, y_0y_0^0\})$$ $$\cup \{a_1c_1\}\cup\bigcup\limits_{i=1}^{\frac{k-1}{2}}\{a_{2i}x_{2i+1}^0, a_{2i+1}y_{2i+1}^0\}\cup\bigcup\limits_{i=2}^{k-1}\{ v_{i-2}c_i\}. $$(See Figure~\ref{F8}). Clearly, $E_0\cup E_2$ is an even cycle $C$.

\begin{figure}[ht]
\begin{center}
\vskip-0.4cm
\includegraphics[scale=0.3]{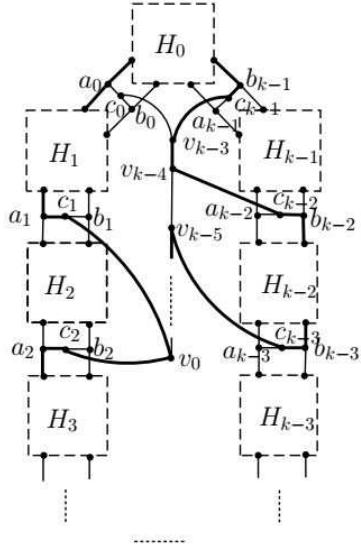}
\vskip-10cm
\end{center}
\caption{$\{G; G_0, G_1, \ldots,G_{k-2}, G_{k-1}\}$}\label{F8}
\end{figure}

 If $i$ is odd or $i=0$, by $x_iy_i\in E_i^{2}$, there exists a maximal path containing only 2-degree vertices as inter vertices in the graph $G_i\backslash E_i^0$, say $u_i^0\ldots y_ix_i\ldots u_i^1$, which corresponds to an edge $u_i^0u_i^1$ in the graph $\overline{G_i\setminus E_i^0}$ (see Figure~\ref{F2}); By $x_i^0x_i, y_i^0y_i\in E_i^{0}$, there exists two distinct maximal path containing only 2-degree vertices as inter vertices in the graph $G_i\setminus E_i^2$, say $u_i^2\ldots x_i^0x_ix_i^1\ldots u_i^3$ and $u_i^5\ldots y_i^0y_iy_i^1\ldots u_i^4$ which correspond to the edges $u_i^2u_i^3$ and $u_i^4u_i^5$ respectively in the graph $\overline{G_i\setminus E_i^2}$ (see Figure~\ref{F3}).

 If $i$ is even and $i\neq 0$, since $x_{i}, y_i\not\in V(E_i^{2})$, there exist four maximal paths containing only 2-degree vertices as inter vertices in the graph $G_i\setminus E_i^2$, say $u_i^0\ldots x_i^1x_i$ (maybe $u_i^0=x_i^1$), $u_i^1\ldots x_i^0x_i$ (maybe $u_i^1=x_i^0$), $u_i^2\ldots y_i^1y_i$ (maybe $u_i^2=y_i^1$) and $u_i^3\ldots y_i^0y_i$ (maybe $u_i^3=y_i^0$), which correspond to the edges $u_i^0x_i$, $u_i^1x_i$, $u_i^2y_i$ and $u_i^3y_i$, respectively, in the graph $\overline{G_i\setminus E_i^2}$. (See Figure~\ref{F4}).
Similarly, by $x_{i}, y_i\not\in V(E_i^{0})$, there exist four maximal paths containing only 2-degree vertices as inter vertices in the graph $G_i\setminus E_i^0$, say $u_i^4\ldots x_i^1x_i$ (maybe $u_i^4=x_i^1$), $u_i^5\ldots x_i^0x_i$ (maybe $u_i^5=x_i^0$), $u_i^6\ldots y_i^1y_i$ (maybe $u_i^6=y_i^1$) and $u_i^7\ldots y_i^0y_i$ which correspond to the edges $u_i^4x_i$, $u_i^5x_i$, $u_i^6y_i$ and $u_i^7y_i$,respectively, in the graph $\overline{G_i\setminus E_i^0}$.

\vskip0.1cm

 If $k=1$, then $H=G_0$ which has a Fulkerson-cover.

 If $k\geq 2$, we will prove $\overline{H\setminus E_0}$ and $\overline{H\setminus E_2}$ are 3-edge colorable in the following.

From the construction of $H$, one has that $\overline{H\setminus E_0}$ (see Figure~\ref{F9}) is
$$
\begin{array}{l}
(\overline{G_0\setminus E_0^2}-\{u_0^2u_0^3, u_0^4u_0^5\})\cup(\overline{G_1\setminus E_1^0}-u_1^0u_1^1)\cup Q_0\cup Q_1
\cup\bigcup_{j=1}^{\frac{k-1}{2}}\{u_{2j}^2b_{2j}, u_{2j+1}^3b_{2j}, u_{2j+1}^2u_{2j}^3\}\\
\cup\bigcup\limits_{j=1}^{\frac{k-3}{2}}\{u_{2j+1}^4b_{2j+1}, u_{2j+1}^5u_{2j+2}^1, u_{2j+2}^0b_{2j+1}, b_{2j}b_{2j+1}\}
\cup\{u_0^5c_0, u_0^4b_0, b_0c_0, b_{k-1}c_0, u_1^1b_0, u_1^0b_1, u_2^0b_1, u_2^1b_1\}.
\end{array}$$
Where, $Q_0=\bigcup\limits_{j=1}^{\frac{k-1}{2}}(\overline{G_{2j}\setminus E_{2j}^2}-\{x_{2j}, y_{2j}\})$ and
$Q_1=\bigcup\limits_{j=1}^{\frac{k-3}{2}}(\overline{G_{2j+1}\setminus E_{2j+1}^2}-\{u_{2j+1}^2u_{2j+1}^3, u_{2j+1}^4u_{2j+1}^5\})$

And $\overline{H\setminus E_2}$ (see Figure~\ref{F9}) is
$$\begin{array}{l}
(\overline{G_1\setminus E_1^2}-\{u_1^2u_1^3, u_1^4u_1^5\})\cup(\overline{G_0\setminus E_0^0}-u_0^0u_0^1)\\
\cup\bigcup\limits_{j=1}^{\frac{k-1}{2}}(\overline{G_{2j}\setminus E_{2j}^0}-\{x_{2j}, y_{2j}\})\cup\bigcup\limits_{j=1}^{\frac{k-3}{2}}(\overline{G_{2j+1}\setminus E_{2j+1}^0}-\{u_{2j+1}^0u_{2j+1}^1\})\\
\cup\bigcup_{j=1}^{\frac{k-1}{2}}\{u_{2j}^6b_{2j}, u_{2j+1}^1b_{2j}, u_{2j}^7b_{2j}\}\cup\bigcup\limits_{j=1}^{\frac{k-3}{2}}\{u_{2j+1}^0b_{2j+1}, u_{2j+2}^4b_{2j+1}, u_{2j+2}^5b_{2j+1}\}\\
\cup\{u_1^5u_2^5, u_1^4b_1, u_2^4b_1, u_1^2c_0, u_1^3b_0, b_0c_0, u_0^0b_0, c_0b_1\}.
\end{array}$$

If $i$ is odd or $i=0$, because $\overline{G_i\setminus E_i^0}$ is 3-edge colorable, there exists a 2-factor $C_i^0$ such that each component is an even circuit and $u_i^0u_i^1$ is not in the 2-factor $C_i^0$. Because $\overline{G_i\setminus E_i^2}$ is 3-edge colorable, there exists a 2-factor $C_i^2$ such each component is an even circuit and $u_i^2u_i^3, u_i^4u_i^5$ are in the 2-factor $C_i^2$.

If $i$ is even and $i\neq 0$, because $\overline{G_i\setminus E_i^2}$ is 3-edge colorable, there exists a 2-factor $C_i^2$ such each component is an even circuit and two paths with length two $u_i^0x_iu_i^1$ and $u_i^2y_iu_i^3$ are in the 2-factor $C_i^2$. Similarly, because $\overline{G_i\setminus E_i^0}$ is 3-edge colorable, there exists a 2-factor $C_i^0$ such each component is an even circuit and $u_i^4x_iu_i^5, u_i^6y_iu_i^7$ are in the 2-factor $C_i^0$.

Then $\overline{H\setminus E_0}$ (see Figure~\ref{F9}) has a 2-factor:
$$\begin{array}{l}
(C_0^2-\{u_0^4u_0^5, u_0^2u_0^3\})\cup C_1^0\cup\bigcup\limits_{j=1}^{\frac{k-3}{2}}(C_{2j+1}^2-\{u_{2j+1}^2u_{2j+1}^3, u_{2j+1}^4u_{2j+1}^5\})\\
\cup \bigcup\limits_{j=1}^{\frac{k-1}{2}}(C_{2j}^2-\{x_{2j}, y_{2j}\})\cup\bigcup\limits_{j=1}^{\frac{k-1}{2}}\{b_{2j}u_{2j+1}^3, b_{2j}u_{2j}^2, u_{2j}^3u_{2j+1}^2\}\cup \bigcup\limits_{j=1}^{\frac{k-3}{2}}\{b_{2j+1}u_{2j+1}^4,\\
b_{2j+1}u_{2j+2}^0, u_{2j+1}^5u_{2j+2}^1\}\cup \{b_0u_0^4, b_1u_2^0, b_1u_2^1, b_0c_0, u_0^5c_0\}
\end{array}$$ and each component is an even circuit.

\begin{figure}[ht]
\begin{center}
\vskip-0.5cm
\includegraphics[scale=0.55]{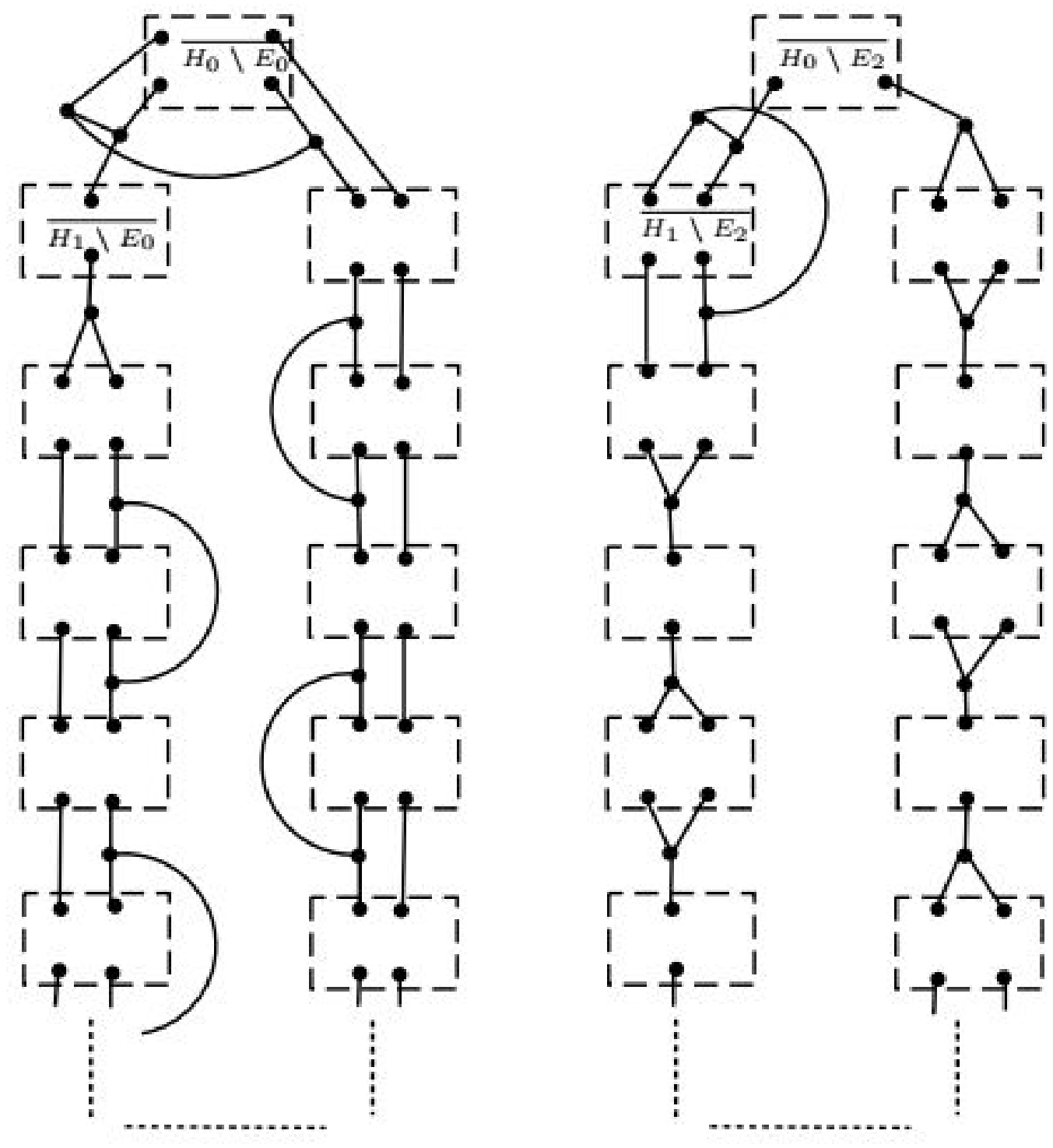}
\vskip-10cm
\end{center}
\caption{$\overline{H\setminus E_0}$ and $\overline{H}\setminus E_2$\\ $H=\{G; G_0, G_1, \ldots, G_{k-2}, G_{k-1}\}$}\label{F9}
\end{figure}

$\overline{H\setminus E_2}$ has a 2-factor:
$$\begin{array}{l}
 C_0^0\cup (C_1^2-\{u_1^2u_1^3, u_1^4u_1^5\})\cup\bigcup\limits_{j=1}^{\frac{k-1}{2}}(C_{2j}^0-\{x_{2j}, y_{2j}\})\cup \bigcup\limits_{j=1}^{\frac{k-3}{2}}C_{2j+1}^0\cup\bigcup_{j=1}^{\frac{k-1}{2}}\{u_{2j}^6b_{2j}, u_{2j}^7b_{2j}\}\\ \cup\bigcup\limits_{j=1}^{\frac{k-3}{2}}\{u_{2j+2}^4b_{2j+1}, u_{2j+2}^5b_{2j+1}\}\cup \{c_0u_1^2, b_0c_0, b_0u_1^3, u_1^5u_2^5, u_1^4b_1, u_2^4b_1\}.\end{array}$$
 And each component is an even circuit.

 So $\overline{H\setminus E_0}$ and $\overline{H\setminus E_2}$ are 3-edge colorable. Therefore $E_0$ and $E_2$ are the desired matchings of Lemma~\ref{lemma-1} and $H=\{G; G_0, G_1,$ $\ldots,G_{k-2},$ $ G_{k-1}\}$ has a Fulkerson-cover.

\end{proof}

From Theorem~\ref{thm-2} and Theorem~\ref{thm-3}, we get the Theorem~\ref{Main thm} that every graph in $M_{0,1,2,\ldots,(k-1)}$ has a Fulkerson-cover.

\section*{Acknowledgments}
This work was supported by the NSFC (No.11371052), the Fundamental Research Funds for the Central Universities (Nos. 2016JBM071, 2016JBZ012), the $111$ Project of China (B16002).

\end{document}